\documentclass[11pt]{article}
\usepackage{latexsym,color,amsmath,amsthm,amssymb,amscd,amsfonts,mathtools}

\usepackage[a4paper, left=1in, right=1in, top=1in, bottom=1in]{geometry}

\newtheorem{lemma}{Lemma}

\newtheorem{fact}[lemma]{Fact}
\newtheorem{theorem}[lemma]{Theorem}
\newtheorem{definition}[lemma]{Definition}

\newtheorem{lem}[lemma]{Lemma}
\newtheorem{prop}[lemma]{Proposition}
\newtheorem{cor}[lemma]{Corollary}

\newtheorem{rem}[lemma]{Remark}

\newtheorem*{remark*}{Remark}

\def\S{{\cal S}}

\def\qed{\hfill $\vcenter{\hrule height .3mm
		\hbox {\vrule width .3mm height 2.1mm \kern 2mm \vrule width .3mm
			height 2.1mm} \hrule height .3mm}$ \bigskip}

\def\to{\rightarrow}

\newcommand{\iprod}[2]{\left\langle #1,#2 \right\rangle} 

\def\RR{\mathbb{R}}

\def\K{{\cal K}}

\DeclarePairedDelimiter{\parens}()

\title{Isometries of the class of ball-bodies}
\author{Shiri Artstein-Avidan, Arnon Chor, Dan Florentin \thanks{The first and second named authors are supported in part by the ERC under the European Union’s Horizon 2020 research and innovation programme (grant agreement no. 770127), by ISF grant Number 784/20, and by the Binational Science Foundation (grant no. 2020329).}}
\date{\today}

\begin{document}
\baselineskip=17pt

\maketitle 

\renewcommand{\thefootnote}{}



\renewcommand{\thefootnote}{\arabic{footnote}}
\setcounter{footnote}{0}

\vspace{-45pt}

\begin{abstract}
	We characterize the surjective isometries, with respect to the Hausdorff distance, of the class of bodies given by intersections of Euclidean unit balls. We show that any such isometry is given by the composition of a rigid motion with either the identity or the $c$-duality mapping, which maps a body in this class to the intersection of Euclidean unit balls centered at points in the body. 
\end{abstract}
Let ${\K}^n$ stand for the set of all convex bodies (i.e. compact, convex, non-empty sets) in $\RR^n$, which is equipped with the Euclidean inner product $\iprod{\cdot}{\cdot}$ and induced norm $\|\cdot\|_2$. Let $\delta$ stand for Hausdorff distance on ${\K}^n$, given by 
\[
\delta(K_0, K_1) = \inf \{\lambda>0: K_0 \subseteq K_1 + \lambda B_2^n, K_1 \subseteq K_0 + \lambda B_2^n\},\]
where $B_2^n$ is the Euclidean unit ball.

The class $\S_n\subseteq {\K}^n$ consists  of  convex bodies which are intersections of unit Euclidean balls in $\RR^n$. It can also be characterized at the class of summands of $B_2^n$, namely bodies $K\in {\K}^n$ for which there exists $L\in {\K}^n$ with $K + L = B_2^n$.
This class is relevant to the study of various open problems in convex analysis, such as the Kneser-Poulsen conjecture (see \cite{bezdek2008kneser}), the Blaschke-Lebesgue problem (see \cite{martini2019bodies}), isoperimetric-type inequalities (see \cite{drach2023reverse}), and more, see \cite{AF-perrint} for additional references.
In \cite{AF-perrint}, we studied this class intensively, in particular noting that up to rigid motions, there is a unique order reversing bijection defined on this set, mapping a body $K\in \S_n$ to the unique $L\in \S_n$ with $K-L = B_2^n$. We denote this $L$ by $K^c$ and call this mapping the $c$-duality on $\S_n$. In terms of support functions (defined as $h_K(u)=\sup_{x\in K}\iprod{x}{u}$ for $u \in \RR^n$), this means $h_{K^c} (u) = 1- h_K(-u)$. In particular, the $c$-duality mapping is an isometry with respect to the Hausdorff metric, which can be realized as the $L^\infty$ norm on the support functions defined on the unit sphere $\partial B_2^n$. We note that for any $K \in \S_n$, $K^{cc} = K$. In geometric terms one can explicitly write this duality as 
\[ K^c = \bigcap_{x\in K} (x+B_2^n).\]

In this note we show that 
up to rigid motions,  there are precisely two bijective isometries on the class $\S_n$, one is the identity and the other is this $c$-duality $K\mapsto K^c$. 

\begin{theorem}\label{thm:main}
	Let $T:\S_n\to \S_n$ be a bijective isometry with respect to the Hausdorff distance. Then there exists a rigid motion $g:\RR^n \to \RR^n$ such that either $TK = gK$ for all $K\in \S_n$ or $TK = gK^c$ for all $K\in \S_n$. 
\end{theorem}

We rely heavily on the ideas from the paper of Schneider \cite{Schneider1975} where he showed that the only surjective isometries of ${\K}^n$ with respect to $\delta$ are those given by rigid motions. 
We mention also the work of Gruber and Lettl \cite{GruberLettl1980}  who showed that without assuming surjectivity, an isometry $F$ of $\K^n$ takes the form  $F(K)=gK+L$ for some rigid motion $g$ and some $L \in \K^n$. 

We   make use of the following easily verifiable fact, which  holds since elements of $\S_n$ are intersections of unit balls.
\begin{fact}\label{fact:unique}
    Let $K \in \S_n$ and $u \in \partial B_2^n$. Then there exists a unique point $x\in \partial K$ such that $u$ is a normal vector to $K$ at $x$.
\end{fact}
 
 We also use the immediate fact that for $A,B,C\in \K^n$ and $\lambda \in [0,1]$,
\[ (1-\lambda) \delta (A,B)  = \delta ((1-\lambda)A+\lambda C, (1-\lambda) B+\lambda C).\]
(We use averages since if $A,B,C$ belong to $\S_n$ then so do their averages.)
In particular we obtain with the above formula that given $K_0,K_1\in \K^n$, the ``line'' $\lambda\mapsto K_\lambda =(1-\lambda)K_0 + \lambda K_1$ for $\lambda \in [0,1]$ is a geodesic curve, namely 
\[ 
\delta (K_0, K_\lambda) + \delta (K_\lambda, K_1 ) = \delta (K_0, K_1).
\]

An isometry maps geodesics to geodesics, which would have been extremely  helpful had it been the case that these weighted Minkowski averages $(K_\lambda)_{\lambda \in [0,1]}$ are the only geodesics between a pair of bodies. However, we can easily find $K_0, K_1$ and some $T\neq \frac{K_0 + K_1}{2}$ with
\[
    \delta (K_0, T) = \delta (K_1, T) = \frac12 \delta (K_0, K_1).
\]
For example consider $K_0 = \frac12 B_2^n$ and $K_1 = x + \frac12 B_2^n$ for some $x\in \RR^n$ with $\|x\|_2 = 1$. Clearly $\delta (K_0, K_1) = 1$. Any $T$ which includes the segment $[0, x]$ and is included in the intersection $B_2^n \cap (x + B_2^n)$, is in distance $\frac12$ to both $K_0$ and $K_1$. 
Since there are many such bodies, we see that  the ``midpoint'' is not unique. Having fixed one such point, we can complete it to a geodesic using the aforementioned averages construction, i.e. take its weighted averages with both endpoints. Since Hausdorff distance is sub-additive we get many different geodesics. 
In this example we find many midpoints even in $\S_n$, not just in $\K_n$; any body lying between the so-called ``$1$-lens'' $[0,x]^{cc}$ and the intersection $B_2^n \cap (x + B_2^n)$. 
However, 
had we started with $1$-balls instead of $\frac12$-balls, there would only be one midpoint within $\S_n$, although many in $\K^n$, as we shall see below. This gives a hint on our method of proof - detecting those pairs of bodies which admit unique geodesics. 

\begin{definition}
    A pair $(K_0,K_1) \in \S_n \times \S_n$  is called $\S_n$-cute if  $M = \frac{K_0+K_1}{2} \in \S_n$ is the unique body in $\S_n$ satisfying 	
	\[ \delta (K_0, M) = \frac12 \delta (K_0, K_1) \qquad {\rm and}\qquad\delta (K_1, M) =\frac12 \delta (K_0, K_1).\]
\end{definition}

\begin{lem}\label{lem:cuteness-and-isometry}
	Let $T: \S_n \to \S_n$  
	be a bijective isometry with respect to the Hausdorff distance and let $(K_0, K_1)$ 
be an $\S_n$-cute pair. Then so is the pair $(TK_0, TK_1)$. Moreover, 
	\[ T\parens*{\frac12 K_0 + \frac12 K_1}  = \frac12 (T( K_0) +T( K_1)). \]
\end{lem}

\begin{proof}
	Let $T: \S_n \to \S_n$ be a bijective isometry with respect to the Hausdorff distance and let $(K_0,K_1)$ be an $\S_n$-cute pair. Since $T$ is surjective, a midpoint of $TK_0$ and $TK_1$ must be the image of a midpoint of $K_0$ and $K_1$, and since only one such midpoint exists for the latter, the same is true for the former. For the ``moreover'' we note that the Minkowski average of two bodies (with equal weights) is always a midpoint of the two with respect to Hausdorff distance. 
\end{proof}

\begin{lem}\label{lem:geod1}
	For two points $x,y\in \RR^n$, the pair $(\{x\}, \{y\})$ is an $\S_n$-cute pair, and the pair $(x+B_2^n,  y+B_2^n)$ is an $\S_n$-cute pair. 
\end{lem}

\begin{rem}
    As the proof shall demonstrate, a pair of points $(\{x\},\{y\})$ is not only $\S_n$-cute but in fact $\K^n$-cute (with the natural definition - this is Schneider's \cite{Schneider1975} definition of ``ein zentriertes Paar''). However, it is easy to check that a pair of unit balls $(x+B_2^n, y+B_2^n), x\neq y$ is not $\K^n$-cute. 
\end{rem}

\begin{proof}[Proof of Lemma~\ref{lem:geod1}]
	For convenience we denote $x$ instead of $\{x\}$. Let $x,y \in \RR^n$ be given and denote $d = \delta(x,y) = \|x-y\|_2$.
    If $K\in \K^n$ satisfies
    \[
        \delta (x, K) = \frac{d}2 \qquad {\rm and} \qquad \delta (y, K) = \frac{d}2 ,
    \]
    then $K\subseteq \left( x + \frac{d}{2} B_2^n \right) \cap \left( y + \frac{d}{2} B_2^n \right) = \frac{x+y}{2}$. 
    This shows that $(x,y)$ is an $\S_n$-cute pair. 
	
    To show that $(x+B_2^n, y+B_2^n)$ forms an $\S_n$-cute pair we use Lemma \ref{lem:cuteness-and-isometry} with the bijective isometry $K\mapsto K^c$. Since $x^c = x+B_2^n$ and $y^c = y + B_2^n$ they must be an $\S_n$-cute pair. 
\end{proof}

We next show that when the distance between two bodies in $\S_n$ is not small, the pairs discussed in Lemma \ref{lem:geod1} are actually the only  $\S_n$-cute pairs.

\begin{prop}\label{prop:cute-far-are-balls-and-points}
	Assume $(K_0,K_1) \in \S_n \times \S_n$ is an $\S_n$-cute pair at distance $\delta(K_0, K_1) \ge 4$. Then either both $K_0,K_1$ are two points or  $K_0,K_1$ are two Euclidean unit balls. 
\end{prop}

The proof will make use of  several lemmas. For these, we define a \textit{lens} to be a nonempty instersection of two balls of radius 1. A lens is \textit{proper} if it is neither a point nor a ball. The first lemma says that a (far enough) $\S_n$-cute pair must have a lens as its midpoint.

\begin{lem}\label{lem:av-is-lens}
	Assume $K_0,K_1\in \S_n$ are an $\S_n$-cute pair with $\delta(K_0,K_1)\ge 4$.  Then $\frac12 (K_0+K_1)$ is a lens. 
\end{lem}

\begin{proof}
	Let $(K_0,K_1)$ be an $\S_n$-cute pair with $d := \delta(K_0,K_1) \geq 4$. Clearly $M = \frac12 (K_0+K_1)$ is at distance $\frac{d}{2}$ from both $K_0$ and $K_1$. 
    Pick $u \in \partial B_2^n$ which maximizes $h_{K_1}(u)-h_{K_0}(u)$ and $w \in \partial B_2^n$ which maximizes $h_{K_0}(w) - h_{K_1}(w)$. Denote $d_0 = h_{K_1}(u) - h_{K_0}(u)$ and $d_1 = h_{K_0}(w) - h_{K_1}(w)$. Note that $d = \max\{d_0,d_1\}$, and additionally, $d_0,d_1 > 0$ since if $d_0 \leq 0$ then $K_1 \subseteq K_0$, but as $K_0$ is contained in some unit ball, this contradicts $d \geq 4$. A similar reasoning gives $d_1 > 0$.

    Denote by $x_{K_0}(u), x_M(u), x_{K_1}(u)$ the unique points of $K_0, M, K_1$ respectively that have $u$ as a normal (these are unique by Fact~\ref{fact:unique}). We claim that $x_{K_0}(u), x_{K_1}(u)$ and $x_M(u)$ must be aligned, that is, 
	\[
        x_{K_0}(u) + \frac{d_0}{2} u =  x_M(u) = x_{K_1}(u) - \frac{d_0}{2} u .
    \]
	Indeed, since $u$ is a maximizer of $h_{K_1} - h_{K_0}$ on $\partial B_2^n$, by the Lagrange multipliers theorem we get that $\nabla (h_{K_1} - h_{K_0})(u)  = \lambda u$ for some constant $\lambda$. Since $\nabla h_{K_i} (u) = x_{K_i}(u)$ we see that $x_{K_1}(u) - x_{K_0}(u) = \lambda u$ but taking scalar product with $u$ we find that $d_0 = h_{K_1}(u) - h_{K_0}(u) = \lambda$, and since the support set of the average is the average of the support sets, we are done. Similarly we get
    \[
        x_{K_0}(w) - \frac{d_1}{2} w = x_M(w) = x_{K_1}(w) + \frac{d_1}{2} w ,
    \]
    where $x_{K_0}(w), x_M(w), x_{K_1}(w)$ are the unique points of $K_0,M,K_1$ respectively with normal $w$.
	
	Next we intersect the balls tangent to $M$ at $x_M(u)$ and $x_M(w)$, i.e., we construct the lens
	\[
        L = (x_M(u)-u + B_2^n) \cap (x_M(w)-w+B_2^n) .
	\]
	Since $M \in \S_n$, we have $M \subseteq L$. Therefore, since $\delta(M,K_0) = \delta(M,K_1) = \frac{d}{2}$, we have that $K_0 \cup K_1 \subseteq M + \frac{d}{2} B_2^n \subseteq L + \frac{d}{2} B_2^n$.
	
	On the other hand, since $0 < d_0 \leq d$ and $d \geq 4$, $\parens*{\frac{d_0}{2} - 1}u \in \parens*{\frac{d}{2} - 1} B_2^n$ and so $(\frac{d_0}{2} - 1)u + B_2^n \subseteq \frac{d}{2} B_2^n$. As $x_M(u) = x_{K_0}(u) + \frac{d_0}{2} u$, we have
    \[
        L \subseteq x_M(u) - u + B_2^n = x_{K_0}(u) + \parens*{\frac{d_0}{2} - 1} u + B_2^n \subseteq x_{K_0}(u) + \frac{d}{2} B_2^n \subseteq K_0 + \frac{d}{2} B_2^n .
    \]
    Similarly we can show that $L \subseteq x_M(w) - w + B_2^n \subseteq K_1 + \frac{d}{2} B_2^n$. We thus see that the lens $L$ is indeed in distance $\frac{d}{2}$ from both $K_0$ and $K_1$. Since the pair $(K_0,K_1)$ was assumed to be $\S_n$-cute, we conclude that $M=L$, which proves that the Minkowski average of $K_0$ and $K_1$ is a lens.
\end{proof}

It is quite hard for a Minkowski average of two bodies in $\S_n$ to be a lens. In fact, the lemmas below show that the two bodies averaged must be translations of the same lens in opposite directions. We prove this in two steps. The first step appears in the paper \cite{AF-perrint}, and we include the short proof here for completeness. The proof does use two non trivial facts about the class $\S_n$ and the $c$-transform, which are proved in \cite{AF-perrint}. 

\begin{lem}\label{lem:average-cap}
	Let $K_0, K_1\in\S_n$ and set $M = \frac12(K_0+K_1)$. Assume that $B_2^n$ is a supporting ball of $M$, and that $A\subseteq \partial M$ is a spherically convex subset of the sphere $\partial B_2^n$. Then there exists $x_0 \in \RR^n$ with $A+x_0\subseteq \partial K_0$ and $A-x_0 \subseteq \partial K_1$.
\end{lem}

\begin{proof}
    Let the notations be as in the statemente of the Lemma. Since $A\subseteq \partial M \cap \partial B_2^n$, and $M \subseteq B_2^n$ by assumption, every $u \in A$ is also the unique point of $M$ that has $u$ itself as a normal. By Fact~\ref{fact:unique} and by the definition of $M = \frac{K_0 + K_1}{2}$, every $u\in A$ can be written uniquely as $u = \frac{\chi_0(u) + \chi_1(u)}{2}$ where $\chi_0(u)$
    is the unique point on $K_0$ where $u$ is the normal, and $\chi_1(u)$
    is the unique point on $K_1$ with $u$ as the normal.
	
	Since $K_0, K_1\in \S_n$, we know that $\chi_0(u)-u \in \partial K_0^c$ and $\chi_1(u) - u \in \partial K_1^c$  (see~\cite[Lemma 1.24]{AF-perrint}). 
	Moreover, as $\chi_0(u) +\chi_1(u) = 2u$, we get that $\chi_0(u)-u = -(\chi_1(u)-u)$.
	Using the fact that the $c$-transform commutes with averages (see~\cite[Theorem 1.19]{AF-perrint}), $M^c = \frac12(K_0^c + K_1^c)$, and we see that for all $u,v\in A$ it hold that $\frac{\chi_0(u) - u + \chi_1(v)-v}{2}\in M^c$, which can equivalently be written as $\frac{\chi_0(u) - u - \chi_0(v)+v}{2}\in M^c$. 
	
	Assume by way of contradiction that for some $u,v\in A$ we have that $\chi_0(u)-u\neq \chi_0(v)-v$. Then all four averages of $\chi_0(u)-u$ or $\chi_0(v)-v$ with $\chi_1(u)-u = u-\chi_0(u)$ or $\chi_1(v)-v= v-\chi_0(v)$ lie in $M^c$, namely 
	\[
          0, \frac{\chi_0(u)-\chi_0(v)+v - u}{2}, \frac{\chi_0(v)-\chi_0(u) -u + v}{2} \in M^c .
    \]
    Since $B_2^n$ is a supporting ball of $M$, $0 \in \partial M^c$, and so there is some $\nu \in \partial B_2^n$ normal to $M^c$ at 0. As $0$ is the average of the points $\frac{\chi_0(u) - \chi_0(v) + v - u}{2}$ and $\frac{\chi_0(v) - \chi_0(u) - u + v}{2}$, we have that $\nu$ is normal to $M^c$ at $\frac{\chi_0(u) - \chi_0(v) + v - u}{2}$ as well, in contradiction to the fact that every direction has a single supporting point in $M^c \in \S_n$.
	
	We conclude that the points $\chi_0(u)$ are all of the form $x_0 +u$ for some fixed $x_0 \in \RR^n$, which means that $x_0 + A \subseteq \partial K_0$ and similarly we get  that $-x_0 + A \subseteq \partial K_1$.
\end{proof}

\begin{cor}\label{cor:both-are-lens}
	Let $K_0, K_1\in\S_n$ and assume $M = \frac12(K_0+K_1)$ is a lens. Then there exists $x_0\in \RR^n$ with $K_0 = M+x_0 $ and $K_1  = M-x_0$.
\end{cor}

\begin{proof}
    Let $K_0,K_1 \in \S_n$ and $M = \frac12 (K_0 + K_1)$. Consider first the case where $M$ is a point or a ball of radius 1. If $M = \{y\}$ is a point, we immediately conclude that $K_0,K_1$ must also be points, and that there exists $x_0 \in \RR^n$ such that $K_0 = \{y+x_0\}, K_1 = \{y-x_0\}$. If $M = y + B_2^n$ for some $y \in \RR^n$, using Lemma~\ref{lem:cuteness-and-isometry} with the $c$-duality isometry we see that $\{y\} = (y + B_2^n)^c = \frac12 K_0^c + \frac12 K_1^c$, which implies that there is some $x_0 \in \RR^n$ such that $K_0^c = \{y+x_0\}$ and $K_1^c = \{y-x_0\}$, and therefore $K_0 = y+B_2^n + x_0$ and $K_1 = y+B_2^n - x_0$. 
    Thus we henceforth assume that $M$ is neither a point nor a ball of radius 1.
    
	Otherwise, since $M$ is a lens, we may assume $M = (z + B_2^n) \cap (-z + B_2^n)$ so that $A_1 = \partial M \cap \{x: \iprod{x}{z}\le 0\}$ and $A_2 = \partial M\cap \{x: \iprod{x}{z}\ge 0\}$ are caps on its boundary. By Lemma \ref{lem:average-cap} there are some $x_1, x_2 \in \RR^n$ with $A_1 + x_1 \subseteq 
	\partial K_0$,  $A_1 - x_1 \subseteq 
	\partial K_1$,  $A_2 + x_2 \subseteq 
	\partial K_0$,  $A_2 - x_2 \subseteq 
	\partial K_1$.
	We will show that $x_1 = x_2$, which implies $K_0 = M + x_1$ and $K_1 = M - x_1$.

	Take a point $z_0\in \partial M$ 
	which belongs to both caps, and so its normal cone includes both $z_0-z$ and $z_0 + z$. 
	This point $z_0$ is the average of $z_0 +x_1 \in K_0$ and $z_0 - x_1 \in K_1$, and also the average of $z_0 + x_2 \in K_0$ and $z_0 - x_2 \in K_1$. 
	
	The normal cone to $M$ at $z_0$ constitutes of all normals in the spherical convex hull of $z_0-z$ and $z_0+z$ (i.e. an arc). Take some $u$ which belongs to this convex hull.
	
	Note that 
	\begin{align*}
        h_M(u) &= \iprod{z_0}{u} = \iprod{\frac{1}{2}((z_0 + x_1)+ (z_0 - x_1))}{u} \\
        &= \frac12 \iprod{z_0+x_1}{u} + \frac12 \iprod{z_0-x_1}{u} .
    \end{align*}
    If the point in $K_0$ that has $u$ as a normal is not $z_0 + x_1$, say $y \in K_0$, then $h_{K_0}(u) > \iprod{z_0+x_1}{u}$ and so $\frac12 h_{K_0}(u) + \frac12 h_{K_1}(u) > h_M(u)$, which is a contradiction to $M = \frac12(K_0+K_1)$. Therefore $z_0 + x_1$ is the unique point in $K_0$ where $u$ is the normal. The same reasoning applies to $x_2$, so that $z_0 + x_2$ is once again the unique point of $K_0$ where $u$ is  a normal, which shows $x_1 = x_2$. This shows that $K_0$ and $K_1$ are opposite translates of $M$. 
\end{proof}

\begin{lem}\label{lem:nocuteles}
    Let $x,z\in \RR^n$ with $\|x\|_2 \leq 1 < \|z\|_2$ and denote $L= (B_2^n+x) \cap (B_2^n-x)$. The pair $( L+z, L -z) $ is an $\S_n$-cute pair if and only if either 
	$x = 0$, or $\|x\|_2 = 1$. In other words, two (far enough) translates of a lens are $\S_n$-cute if and only if they are either singletons or unit balls. 
\end{lem}

\begin{proof}
	When $x=0$ or $\|x\|_2=1$ we are in the two cases discussed in Lemma \ref{lem:geod1}, and the pairs are $\S_n$-cute. 
	We thus assume $L$ is a proper lens, namely $x\neq 0$ and $\|x\|_2<1$.
	
    Since $\|x\|_2 \leq 1$, $L$ is nonempty. Consider the unique point $y$ on $\partial L$ which has $z/\|z\|_2$ as a normal. By symmetry, $-y\in \partial L$ is the unique point where the normal is in direction $-z/\|z\|_2$. Note that $y - \frac{1}{\|z\|_2}z + B_2^n \supseteq L$.

    We split the proof to cases based on the dimension $n$ and the point $y \in \partial L$. If $n \geq 3$, or $n = 2$ and $y$ is a smooth point of $\partial L$, the set $M = [-y,y]^{cc} \subsetneq L$. Indeed, if $n \geq 3$ then $M$ has exactly two non-smooth points $\{\pm y\}$, and cannot equal $L$, which has a continuum of non-smooth points; and if $n = 2$ and $y$ is a smooth point of $\partial L$ then again, $y$ is a non-smooth point of $M$ while it is a smooth point of $L$ and thus they cannot be equal. For a more detailed description of various lenses and their duals, see~\cite[Section 7.1]{AF-perrint}.
    Therefore $M \subseteq L+z+\|z\|_2B_2^n$ and $M \subseteq L-z+\|z\|_2B_2^n$. However, since $L\in \S_n$ and $\|z\|_2 > 1$, we see that
    \[
        y + \|z\|_2B_2^n \supseteq y+\parens*{1-\frac1{\|z\|_2}} z + B_2^n \supseteq L+z .
    \]
    Therefore (as $y\in M$) $L+z \subseteq M+\|z\|_2B_2^n$ and similarly $L-z \subseteq M + \|z\|_2B_2^n$, using the ball of radius $\|z\|_2$ centered at $-y$. We conclude that $\delta(M, L+z) = \delta(M, L-z) = \|z\|_2$, and so $M$ is a midpoint of $L+z$ and $L-z$. Since $L$ itself is also a midpoint, the pair $(L+z,L-z)$ is not $\S_n$-cute.

    If $n = 2$ and $y$ is a singular point of $\partial L$, the set
    \[
        P = ((-y+z)+\|z\|_2B_2^n)\cap ((y-z)+\|z\|_2B_2^n) \supsetneq L .
    \]
    Indeed, since $\delta(L,L-z) = \|z\|_2$ we have $L \subseteq y-z + \|z\|_2B_2^n$, and similarly $L \subseteq -y+z + \|z\|_2B_2^n$. Additionally, $L$ is an intersection of two balls of radius 1 and therefore cannot equal $P$, an intersection of two balls of larger radius.
    Thus $P + \|z\|_2B_2^n \supseteq L+z,L-z$. Since $-y+z \in L+z$ and $y-z \in L-z$ we also have $L+z+\|z\|_2B_2^n \supseteq P$ and $L-z+\|z\|_2B_2^n \supseteq P$. Therefore $\delta(P,L+z) = \delta(P,L-z) = \|z\|_2$, so again $P$ and $L$ are two distinct midpoints of $L+z$,$L-z$, which means the pair $(L+z,L-z)$ is not $\S_n$-cute.
\end{proof}

With the above ingredients we prove that an $\S_n$-cute pair of bodies with distance at least $4$ has to be either two balls or two points. 

\begin{proof}[Proof of Proposition \ref{prop:cute-far-are-balls-and-points}]
    Let an $\S_n$-cute pair $(K_0,K_1)$ be given and assume that $\delta(K_0,K_1) \geq 4$. By Lemma \ref{lem:av-is-lens}, their Minkowski average is a lens, and by Corollary \ref{cor:both-are-lens} they are both translates of the same lens. However by Lemma \ref{lem:nocuteles} two translates of a   lens which are in distance at least $4$ are $\S_n$-cute if and only if they are either singletons or balls. 
\end{proof}

Putting together Proposition \ref{prop:cute-far-are-balls-and-points} and Lemma \ref{lem:cuteness-and-isometry} we get the following. 

\begin{lem}\label{lem:point_rigidity}
	If $T:\S_n \to \S_n$ is a bijective isometry then either all points are mapped to points, or all points are mapped to unit balls.
\end{lem}

\begin{proof}
    A pair of points $x,y$ in distance at least $4$ is $\S_n$-cute by Lemma \ref{lem:geod1}. Since $T$ is assumed to be a bijective isometry, by Lemma \ref{lem:cuteness-and-isometry} they must be mapped to an $\S_n$-cute pair $Tx, Ty$ of the same distance, which by Proposition \ref{prop:cute-far-are-balls-and-points} means $Tx,Ty$ are either a pair of points or a pair of Euclidean unit balls. 
    We may then reach any other point by two jumps of length $4$ or more, so we see that all points are mapped to points or all points are mapped to balls. 	
\end{proof}

Finally, we will use the following simple fact, that a body can be reconstructed once we know its Hausdorff distance from all points in $\RR^n$. 

\begin{lem}\label{lem:representations}
	For any $K\in \S_n$,
	\begin{eqnarray*}
		K = \bigcap_{x\in \RR^n} (x + \delta (x, K) B_2^n) .
	\end{eqnarray*}
\end{lem}

\begin{proof}
	Let $K \in \S_n$. By definition of the Hausdorff distance,
    \begin{equation}\label{eq:intersection}
        K \subseteq \bigcap_{x \in \RR^n} (x + \delta(x,K) B_2^n) \subseteq \bigcap_{x \in \RR^n: \delta(x,K) \leq 1} (x + B_2^n) .
    \end{equation}
    Note that for any $x \in \RR^n$, $\delta(x, K) = \inf \{\lambda > 0: K \subseteq x + \lambda B_2^n\}$ since the other inclusion then holds immediately. Thus $x\in K^c$ if and only if $\delta(x, K) \le 1$, so
    \[
        \bigcap_{x \in \RR^n: \delta(x,K) \leq 1} (x + B_2^n) = \bigcap_{x \in K^c} (x+B_2^n) = (K^c)^c = K ,
    \]
    and so in Equation~\eqref{eq:intersection} equality holds throughout.
\end{proof}

We are finally ready to prove Theorem~\ref{thm:main}.

\begin{proof}[Proof of Theorem~\ref{thm:main}]
    Let $T: \S_n \to \S_n$ be a bijective isometry. By Lemma~\ref{lem:point_rigidity} either all points are mapped to points, or all points are mapped to unit balls. In the first case, consider $R = T$, and in the second, consider $R: \S_n \to \S_n$ given by $RK = (TK)^c$. In either case, $R$ is an isometry of $\S_n$ mapping points to points, so its restriction to points must an isometry of $\RR^n$, and these are well know yo be given by rigid motions $g: \RR^n \to \RR^n$. Thus $Q = g^{-1} \circ R: \S_n \to \S_n$ is an isometry which, restricted to points, is the identity.

    Let $K \in \S_n$. Note that for any $x \in \RR^n$ we have $\delta(QK,x) = \delta(K,x)$, so using Lemma~\ref{lem:representations} twice, we get
    \[
        QK = \bigcap_{x \in \RR^n}(x + \delta(x,QK)B_2^n) = \bigcap_{x \in \RR^n}(x + \delta(x,K)B_2^n) = K .
    \]

    Thus either for all $K\in \S_n$, $K = QK = g^{-1}TK$, in which case $TK = gK$, or for all $K\in \S_n$, $K = QK = g^{-1}(TK)^c$, in which case $TK = (gK)^c = gK^c$. This completes the proof.
\end{proof}

\subsection*{Acknowledgements}
We would like to thank the anonymous referees for their helpful comments and remarks. The first and second named authors are supported in part by the ERC under the European Union’s Horizon 2020 research and innovation programme (grant agreement no. 770127), by ISF grant Number 784/20, and by the Binational Science Foundation (grant no. 2020329).

\bibliographystyle{plain}
\bibliography{ref}

\end{document}